\documentclass[sn-mathphys,Numbered]{sn-jnl}

\usepackage{graphicx}%
\usepackage{multirow}%
\usepackage{amsmath,amssymb,amsfonts}%
\usepackage{amsthm}%
\usepackage{mathrsfs}%
\usepackage[title]{appendix}%
\usepackage{textcomp}%
\usepackage{manyfoot}%
\usepackage{booktabs}%
\usepackage{algorithm}%
\usepackage{algorithmicx}%
\usepackage{algpseudocode}%
\usepackage{listings}%
\usepackage{xcolor}
\usepackage{tikz}
\usepackage{xfrac}
\usepackage{units}
\usepackage{caption}
\usepackage{subcaption}
\usepackage[scr=boondox]{mathalfa}
\usepackage{bm}
\newcommand\indicator{{\mathbb{I}}}
\newcommand\E{{\mathbb E}}
\newcommand\convD{{\buildrel {\cal D} \over \longrightarrow}}

\newcommand\normal{{\cal N}}

\theoremstyle{thmstyleone}%
\newtheorem{theorem}{Theorem}
\newtheorem{proposition}[theorem]{Proposition}%

\theoremstyle{thmstyletwo}%
\newtheorem{example}{Example}%
\newtheorem{remark}{Remark}%
\newtheorem{corollary}{Corollary}[section]
\theoremstyle{thmstylethree}%
\newtheorem{definition}{Definition}%

\begin{document}

\title[Article Title]{Applying affine urn models to the global profile of hyperrecursive trees}


\author*[1]{\fnm{Joshua} \sur{Sparks}}\email{ josparks@gwu.edu}

\affil*[1]{\orgdiv{Department of Statistics}, \orgname{The George Washington University}, \orgaddress{\street{} \city{Washington}, \postcode{20052}, \state{DC}, \country{USA}}}


\abstract{Inside the discipline of graph theory exists an extension known as the hypergraph. This generalization of graphs includes vertices along with hyperedges consisting of collections of two or more vertices. One well-studied application of this structure is that of the recursive tree, and we apply its framework within the context of hypergraphs to form hyperrecursive trees, an area that shows promise in network theory. However, when the global profile of these hyperrecursive trees is studied via recursive equations, its recursive nature develops a combinatorial explosion of sorts when deriving mixed moments for higher containment levels. One route to circumvent these issues is through using a special class of urn model, known as an affine urn model, which samples multiple balls at once while maintaining a structure such that the replacement criteria is based on a linear combination of the balls sampled within a draw.

We investigate the hyperrecursive tree through its global containment profile, observing the number of vertices found within a particular containment level, and given a set of $k$ containment levels, relate its structure to that of a similar affine urn model in order to derive the asymptotic evolution of its first two mixed moments. We then establish a multivariate central limit theorem for the number of vertices for the first $k$ containment levels. We produce simulations which support our theoretical findings and suggest a relatively quick rate of convergence.
}

\keywords{ hypergraph, urn model, martingale, limit law, simulation}


\pacs[MSC Classification]{05082, 90B15, 60F05}

\maketitle

\section{Introduction to the Hyperrecursive Tree}\label{sec1}

A graph consists of a set of vertices and a set of edges which either connect two vertices together or one vertex to itself. From this construction, we generalize the definition of a graph so that more than two vertices can be joined together by a given edge. Such a generalization is called a \emph{hypergraph}, where we have hyperedges consisting of collections of one or more vertices.

Many graph classes are worthy candidates to extend into the realm of hypergraphs; one such class of importance is the \emph{recursive tree}. The recursive tree is a tree on $n$ vertices, labeled from 1 to $n$, that is rooted at vertex 1, and such that each unique path from the root downward forms an increasing sequence of vertices to which it crosses. This class has been studied extensively, see \cite{Drmota, QQ, Karonski, Hofri, Smythe} among others. From this construction, we derive a discrete structure that lends itself to the study of hypergraphs: the hyperrecursive tree.  

We define a hyperrecursive tree with parameter (hyperedge size) $\theta$ to be the corresponding hypergraph to the recursive tree, to which we obtain the following construction: We begin with $\theta$ originating vertices, 
all labeled with 0 and contained within the first hyperedge. At each subsequent step, a vertex is added to the structure and labeled by the number corresponding to its step of inclusion. The incoming vertex then chooses $(\theta-1)$ of the existing vertices uniformly at random to co-share a new hyperedge, with all subsets of vertices of size $(\theta-1)$ being equally likely. From this construction, we see that the usual recursive tree follows when $\theta = 2$.
In Figure~\ref{Fig:hyper}, we witness the growth of a hyperrecursive tree at time $n=2$, when $\theta = 3$. Here, the hyperedge
appearing at step~$i$ is labeled $e_i$.
\begin{remark}
When $\theta \ge 3$, the hyperrecursive tree is no longer a tree in the mathematical sense.
We call such a hypergraph by this name only to preserve its historic origin of these structures and how it reduces to a recursive trees when $\theta=2$.
\end{remark}
\bigskip
\begin{figure}[thb]
\begin{center}
\begin{tikzpicture}[scale=0.6]
\node[draw=white] at (-8, 13) {$n=0$};
\node[draw=white] at (0, 13) {$n=1$};
\node[draw=white] at (8, 13) {$n=2$};

\node[draw=white] at (-10.9, 10) {$e_0$};
\node[draw=white] at (-2.9, 10) {$e_0$};
\node[draw=white] at (5.1, 10) {$e_0$};
\node[draw=white] at (7.25, 6.5) {$e_1$};
\node[draw=white] at (-0.75, 6.5) {$e_1$};
\node[draw=white] at (10.8, 8.1) {$e_2$};

\node[draw=white] at (-8, 10) {$0$};
\node[draw=white] at (-9.5, 10) {$0$};
\node[draw=white] at (-6.5, 10) {$0$};

\draw [ultra thick] (-8,10) ellipse (2.5 and 0.8);
\draw [thick] (-9.5,10) circle [radius=0.4];
\draw [thick] (-8,10) circle [radius=0.4];
\draw [thick] (-6.5,10) circle [radius=0.4];

\node[draw=white] at (0, 10) {$0$};
\node[draw=white] at (-1.5, 10) {$0$};
\node[draw=white] at (1.5, 10) {$0$};
\node[draw=white] at (-0.25, 8) {$1$};
\draw [ultra thick] (0,10) ellipse (2.5 and 0.8);
\draw [ultra thick][rotate=-30] (-4.25,8.25) ellipse (1.3 and 2.5);
\draw [thick] (-1.5,10) circle [radius=0.4];
\draw [thick] (0,10) circle [radius=0.4];
\draw [thick] (1.5,10) circle [radius=0.4];
\draw [thick] (-0.25,8) circle [radius=0.4];

\node[draw=white] at (8, 10) {$0$};
\node[draw=white] at (6.5, 10) {$0$};
\node[draw=white] at (9.5, 10) {$0$};
\node[draw=white] at (7.25, 8) {$1$};
\node[draw=white] at (9.25,7.5) {$2$};
\draw [ultra thick] (8,10) ellipse (2.5 and 0.8);
\draw [ultra thick][rotate=-30] (2.5,12.25) ellipse (1.3 and 2.5);
\draw [thick] (6.5,10) circle [radius=0.4];
\draw [thick] (8,10) circle [radius=0.4];
\draw [thick] (9.5,10) circle [radius=0.4];
\draw [thick] (7.25,8) circle [radius=0.4];
\draw [thick] (9.25,7.5) circle [radius=0.4];
\draw[ultra thick] (8.3, 9.3) .. controls (9.4, 12.2)
and (11.5, 12.2) .. (9.797, 6.18);

\draw[ultra thick] (8.32, 9.3) .. controls (5.3, 8.4)
and (5.3, 8) .. (9.838, 6.2);

\end{tikzpicture}
\end{center}
  \caption{A hyperrecursive tree grown in two steps ($\theta =3$).}
  \label{Fig:hyper}
\end{figure}
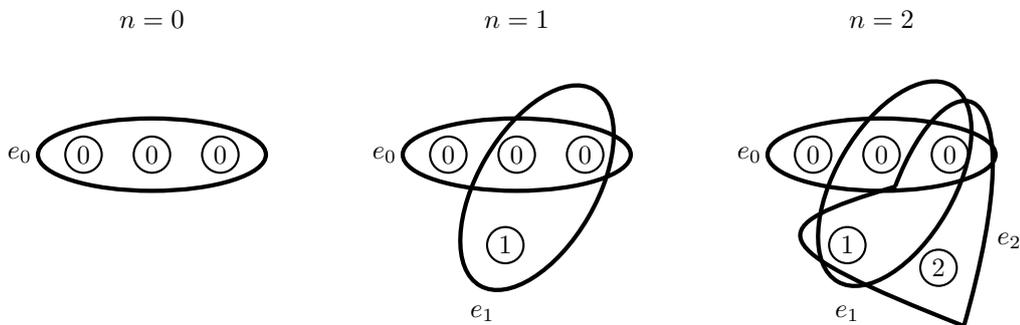
Many interesting properties can be studied from these random structures, two of which include the \emph{local} and \emph{global profiles} of the hyperrecursive tree as studied in \cite{Sparks2}. There, the local profile describes the \emph{level of containment} for a vertex, which evolves at each step and produces asymptotic results which undergo phase changes over time. The global profile (Section \ref{Sec:global}) on the other hand characterizes the number of vertices at a specified containment level.

\section*{Scope}
In \cite{Sparks2}, the global profile of the hyperrecursive tree is explored through the use of recursive equations to generate the means and covariance matrix for the first two containment levels. A central limit theorem for the first containment level is also determined. However, attempts to express means and covariances for higher levels of containment at even the asymptotic level is computationally extensive through this process, and while adaptations of \cite{MahmoudSmythe012} can solve for convergence in distribution for the first three containment levels, this path is also computationally laborious.

One way to circumvent this difficulty is through the use of urn models, specifically through \emph{affine} urn models, which are urn structures with $s\ge1$ drawings with replacement criterion that can be determined by the subset of samples where all balls obtained are of identical color. Often, urns with multi-set drawings become difficult to analyze as we cannot employ classical approaches such as moment methods, analytic combinatorial procedures, or eigenvalue decomposition techniques; however, affine urn models from tenable balanced irreducible schema, as explored in \cite{KubaX, KubaM, Mah2013, Sparks23}, bypass this hurdle through analyzing its \emph{core} matrix $\mathbf{A}$, described more in Section 4.

The rest of the paper proceeds as follows. Section 3 covers some notation and foundational terminology. Section 4 briefly describes affine urn models as well as provides asymptotic results outlined in \cite{Sparks23}. Section 5 then examines the global profile through its relationship to affine urns in order to produce asymptotic results and convergence in distribution. We conclude with simulation results for the first three containment levels in Section 6.

\section{Notation}
\label{Sec:notation}
We begin with some foundational notation. Much of our asymptotic analysis relies on the use of ``Big-Oh'' $(\cal O)$ and ``Small-Oh'' $(o)$ notation. We may use this system via the standard asymptotic notation $o$, $\mathcal O$, but will also need them in a matricial sense, which will be indicated in boldface, so $\mathbf{o}(n)$ and $ \bm{\mathcal O}(n)$ indicate a matrix
of appropriate dimensions with all its entries being $o(n)$ or~$\mathcal{O}(n)$, respectively.

We often use both rising and falling factorials within the context of this research. For rising factorials, we use Pochhammer's symbol, $\langle x \rangle_m$, as the $m$-times rising factorial of a real number $x$, with $\langle x \rangle_m = x(x+1) \ldots (x+m-1)$, and defining $\langle x \rangle_0:=1$. For falling factorials, we use the notation $(x)_{m}=x(x-1)\dots(x-(m-1))$, where $(x)_0:=1$. We define $[k]:=\{1, \dots, k\}$ to be the set of the first $k$ natural numbers.

Many averages and variances contain generalized harmonic numbers. The generalized harmonic number is defined as
$H_n^{(s)} (x)$ $= \sum_{k=1}^n (k+x)^{-s}$, for $n\in \mathbb{Z}_{\ge 0}$ and $s,x\ge 0$. 
The superscript is often dropped when $s=1$, and for any fixed $x$, we have the classic asymptotic approximations:
\begin{equation}
H_n (x) = H_n^{(1)}(x) = \ln (n) - \psi(x) - \frac{1}{x}+\mathcal{O}\Bigl(\frac{1}{n}\Bigr),
\label{Eq:harmonic}
\end{equation}
\begin{equation} 
H_n^{(2)}(x) = \frac{d}{dx}\psi(x+1) + \mathcal{O}\Big(\frac{1}{n}\Big),
\label{Eq:harmonic2}
\end{equation}
where $\psi(.)$ is the digamma function. Note that as $x\to 0$, $- \psi(x) - x^{-1}$ converges to the Euler-Mascheroni constant $\gamma$, defined in \cite{Euler}.

In our asymptotic analysis, we apply the Stirling approximation of the ratio of growing Gamma functions, as detailed in \cite{Tricomi}. Namely, for fixed~$\alpha,\beta \in \mathbb R$,
we have
\begin{equation}
\frac{\Gamma(x+\alpha)} {\Gamma(x+\beta)} \sim x^{\alpha-\beta} \Bigl(1 +\frac{(\alpha-\beta)(\alpha+\beta-1)}{2x}+ \mathcal{O}\Bigl(\frac{1}{x^{2}} \Bigr)\Bigr), \qquad \mbox {as \ } x \to\infty.
\label{Eq:Stirlingapprox}
\end{equation}
This approximation is applicable, even if $\alpha = \alpha(x)$ and $\beta =\beta(x)$ 
grow slowly with $x$. 

We denote a diagonal matrix with the provided entries $a_1, a_2, \dots, a_k,$ as 
\begin{align*}
\mathbf{diag}(a_1, a_2, \dots, a_k) =\small
\begin{pmatrix}
			a_1 & 0&\ldots&0 \cr
                               	0 &a_2&\ldots&0 \cr
			\vdots& \vdots&\ddots&\vdots \cr
                               	 0& 0&\ldots&a_k\cr
\end{pmatrix}.
\label{eq:diag}
\end{align*}

For a square matrix $\mathbf{A}$, the ordering of our eigenvalues is important for our analysis, especially the first two values $\lambda_1$ and $\lambda_2$. In general, we order them such that $\Re{(\lambda_1)}\ge\Re{(\lambda_2)}\ge\dots\ge\Re{(\lambda_k)}$; see \cite{Sparks23} for further guidelines.

Additionally, we may write $\mathbf{A}$ as a {Jordan Decomposition} of matrices in the form 
\begin{equation} 
\label{EQ:JD}
\mathbf{A} =\mathbf{T \mathbf{J} T} ^{-1}, 
\end{equation}
where $\mathbf{T}$ is a square, invertible {transition} matrix and 
\begin{align*}
\mathbf{J}=\small\begin{pmatrix}\displaystyle 
  \mathbf{J}_{\lambda_1} &  \mathbf{0}_{1,2}&\cdots& \mathbf{0}_{1, r-1}&  \mathbf{0}_{1,r}\medskip\\
   \mathbf{0}_{2,1} &\mathbf{J}_{\lambda_2}& \cdots&\mathbf{0}_{2,r-1}&\mathbf{0}_{2,r}\medskip \\
   \vdots &\vdots &\ddots& \vdots &\vdots \medskip \\
   \mathbf{0}_{r-1,1} &  \mathbf{0}_{r-1,2} & \cdots& \mathbf{J}_{\lambda_{r-1}}  &\mathbf{0}_{r-1,r}\medskip \\
   \mathbf{0}_{r,1} & \mathbf{0}_{r,2} &\cdots&\mathbf{0}_{r,r-1}& \mathbf{J}_{\lambda_r}
   \end{pmatrix},
\end{align*}
such that $\{\lambda_\ell\}_{\ell=1}^{r}$ are eigenvalues of $\mathbf{A}$, and each $\mathbf{J}_{\lambda_\ell}$ is a Jordan block of the form
\begin{align*}
\mathbf{J}_{\lambda_\ell}=\small\begin{pmatrix}\displaystyle 
  \lambda_\ell &  1&0&\cdots& 0& 0\\
   0 &\lambda_\ell& 1&\cdots&0&0 \\
   \vdots &\vdots &&\ddots& \vdots &\vdots  \\
   0&0&  0 & \cdots&\lambda_\ell  &1\\
  0&0&0 &\cdots&0&\lambda_\ell
   \end{pmatrix}.
\end{align*}
Such discussion of application of Jordan decomposition can be found in textbook sources such as \cite{Horn}.

With respect to urn models, we let $\mathbf{X}_n$ be our random (row) vector of balls inside of our urn, with counts $X_{n,1}, X_{n,2}, \ldots,X_{n,k}$ after draw $n$. We let $\tau_n$ be the total number of balls in our urn, and so $\tau_n = \sum_{i=1}^{k} X_{n,i}$. We sample $s\ge 1$ balls in a single sample, and will do so here without replacement. We denote the (row) vector of balls sampled (such that $s_i$ balls are sampled of color $i$) by
\begin{align*}
\mathbf{s}=(s_1, s_2, \ldots ,s_k).
\end{align*}

We also wish for our urns to possess the desirable properties of {tenability}, {irreducible}, and to be (${b}$-){balanced} with respect to the replacement criterion. Tenability occurs when our urn scheme can endure indefinite samples without violation of the replacement criterion. Irreducibility can be conceptualized as the ability for all possible colors to be realized regardless of the initial configuration of balls in the urn. An urn is balanced when regardless of the sample drawn, the same total number of balls are added (or removed) from the urn at each stage; for ${b}$-balanced urns, note that $\tau_n = {b} n + \tau_0$.

With this definition of the hyperrecurive tree, we let $\tau_n^{(\theta)}$ be the cardinality
 of the vertex set ({\em size}) of the hyperrecursive tree at time~$n$
(right after the insertion of vertex~$n$).
We sometimes refer to $n$ as the {\em age} of the hyperrecursive tree, and so the total may be expressed as
\begin{equation}
\tau_n^{(\theta)} =  n + \theta.
\label{Eq:total}
\end{equation}

For this model, our sample is taken at time $n$ and drawn without replacement, classified as Model $\cal M$, as denoted by \cite{Chen1, Chen2}. Thus, the vector for the sampled
number of vertices with containment levels $1, \ldots, k$ has a multivariate hypergeometric distribution when conditioned at time $(n-1)$. Further discussion on the multivariate hypergeometric distribution can be found in works such as \cite{Kendall}.

\section{Affine Urn Models}
\label{Affine}
The field of urn models have a rich history in various disciplines as a way to analyze occurrences such as contagion, random walks, fluid diffusion, occupancy problems, and coupon collection; see \cite{Mah2009}. Through its foundational work with two or more ball colors, \cite{Athreya, Bagchi, Flajolet, Smythe2, Janson} and others have developed ways to examine the evolutionary state of an urn via draws of a single ball at each stage. However, when more than a single ball is drawn at a time, many of the traditional techniques develop snags created by larger samples. One class of urn model problems that work around these hurdles are those which classify themselves as \emph{affine}.

The concept of the affine urn {draws} from the construction of the replacement matrix, denoted by $\mathbf{M}$, and from what we call its \emph{core matrix}, $\mathbf{A}$. While~$\mathbf{M}$ may be rectangular (due to multiple drawings), the core matrix $\mathbf{A}$ is a square, $k \times k$ matrix that describes the replacement criterion when all $s$ balls sampled are of the same color. In \cite{KubaX}, each of the $k$ rows in this matrix represent $s$ draws from color $i\in[k]$, and so if we let $\mathbf{e}_i$ be a $1\times k$ row vector with 1 in the $i^{\textnormal{th}}$ position and zeros elsewhere, we may denote the (row) vector of balls sampled (such that $s_i$ balls are sampled of color $i$) by $\mathbf{s}= \sum_{i=1}^{k} s_i \, \mathbf{e}_i=(s_1, s_2, \dots ,s_k)$. Note that tenability follows if the replacement row vector $\bf{a}_{\bf{s}}$ consists of entries $(a_1, a_2,\dots,a_k)$ where $a_i \ge -s_i, 1\le i \le k$. 

The techniques behind our methodology are approached using row vectors, including the random vector $\mathbf{X}_n$. For a draw of $m$ balls, we define the discrete simplex $\Delta$ to be the collection of all possible samples~$\mathbf{s}$, or more precisely,
\begin{align*}
\Delta=\Big\{ \mathbf{s} \,\Big|\,  s_i \ge 0, i \in [k],  \sum_{j=1}^{k} s_j = s \Big\}.
\end{align*}
We set $\mathbf{a}_{\mathbf{s}}=(a_1, a_2, \dots ,a_k)_{\mathbf{s}}$ to be the row vector of ball replacement in our urn for the sample $\mathbf{s}$. To preserve tenability, we ensure that $a_i \ge -s_i$ for all $i\in[k]$ when sampling without replacement. Thus, we treat our replacement matrix $\mathbf{M}$ to be of the form $\mathbf{M}=(\mathbf{a}_{\mathbf{s}})_{\mathbf{s}\in \Delta}$ with \emph{core matrix} $\mathbf{A}=(\mathbf{a}_{s \, \mathbf{e}_i})_{1\le i \le k}$. An example of this construction is displayed in \cite{Sparks23}.

We define $\mathbf{Q}_n:=\mathbf{Q}_{n, \Delta}$ to be the random vector which records the sample \\
$\mathbf{s}=(s_1, s_2, \dots, s_k)$ at draw $n$ from  $\mathbf{X}_{n-1}$, a vector of $k$ colors
$\{X_{n-1, i}\}_{i=1}^{k}$. In an unordered sample without replacement (Model $\mathcal{M}$), this sampling follows a multivariate hypergeometric distribution with probability
\begin{align*} 
\mathbb{P}\big(\mathbf{Q}_{n, \Delta}=\mathbf{s}\,\big|\,\mathbf{X}_{n-1} \bigr)=&\, \begin{pmatrix} s \cr s_1, s_2, \dots, s_r \end{pmatrix} \frac{ (X_{n-1, 1})_{s_1} \cdots (X_{n-1,k})_{s_k}}{(\tau_{n-1})_{s}}.
\end{align*}

With these urn structures, the outcome of the $n^{\textnormal{th}}$ draw derives upon knowledge of the previous draw. Thus, by conditioning on the previous draw, we formulate a \emph{conditional distributional equation} for the random row vector $\mathbf{X}_n$. Conceptually, note that $\mathbf{X}_n$ is the combination of both the count from $\mathbf{X}_{n-1}$, as well as the contribution of the $n^{\textnormal{th}}$ draw. Thus, for a sample of $\mathbf{s}^*$ at draw $n$, we produce the corresponding formula as
\begin{align*}
\mathbf{X}_n= \mathbf{X}_{n-1} + \sum_{\mathbf{s} \in \Delta} \indicator\{\mathbf{Q}_{n, \Delta}=\mathbf{s}^*\} \,\mathbf{a}_{\mathbf{s}}.
\end{align*}

Furthermore, \cite{KubaX} generalizes the concept of dichromatic linear urn models with multiple drawings to two or more colors and derives recursion properties for the model, as given below.
\begin{definition}
A $k$-color urn model with multiple drawings and sample size \\$s\ge 1$ is called linear if the random vector $\mathbf{X}_n$, of ball counts after $n$ draws, satisfies $\E \big[\mathbf{X}_n \,| \,\mathbb{F}_{n-1} \big] = \mathbf{X}_{n-1}\, \mathbf{C}_n$, for a sequence of $k \times k$ real-valued matrices $\{\mathbf{C}_i \}_{i \in \mathbb{N}}$.
\end{definition}

\begin{theorem}
\label{thm:Linearity}
A balanced $k$-color urn model with multiple drawings consisting of unordered samples of size $s$ is linear if and only if $\mathbf{a}_{\mathbf{s}} = \sum_{i=1}^{k} \frac{s_i}{s}\mathbf{a}_{s\mathbf{e}_i}$ for all 
$\mathbf{s}\in \Delta$. Furthermore, the conditional expected values is provided by 
\begin{equation}
\label{EQ: EXQA}
\E \big[\mathbf{X}_n \,\big|\, \mathbb{F}_{n-1} \big]= \mathbf{X}_{n-1} \Bigl(\mathbf{I}_k + \frac{1}{\tau_{n-1}}\mathbf{A}\Bigr).
\end{equation}
\end{theorem}

Note that $\mathbf{a}_{s \,\mathbf{e}_i}$ can be rewritten as $\mathbf{e}_i \mathbf{A}$. Using the sums of these matrices, we see that $\frac{1}{s}\mathbf{s}\,\mathbf{A}$ computes the replacement changes provided for drawing a given sample $\mathbf{s}$. Thus, the recurrence relation for $\mathbf{X}_n$ can be rewritten as
\begin{equation}
\label{EQ:XQA}
\mathbf{X}_n=\mathbf{X}_{n-1} + \frac{1}{s} \mathbf{Q}_n \mathbf{A},
\end{equation}
where $\mathbf{Q}_n$, when conditioned on $\mathbb{F}_{n-1}$, is a multivariate hypergeometric distributions under Model $\mathcal{M}$. Note that $\E\big[\mathbf{Q}_{n}\, \big|\, \mathbb{F}_{n-1}\big] = \frac{s\, \mathbf{X}_{n-1}}{\tau_{n-1}}$, and so from \cite{KubaX}, the conditional expectation $\E\big[\mathbf{X}_{n}\, \big|\, \mathbb{F}_{n-1}\big]$ aligns exactly with that of (\ref{EQ: EXQA}). This result outlines the \emph{linearity} property we desire.

If a square replacement matrix core $\mathbf{A}$ is irreducible, and ${b}$-row-balanced, then its leading eigenvalue $\lambda_1={b}$ is a simple eigenvalue and its principal left and right eigenvectors are denoted as $\mathbf{v}_1$ and $\mathbf{w}_1$, respectively, and scaled such that $\mathbf{w}_1=\mathbf{1}_k$, a vector of $k$ ones. We define its \emph{core index} to be ratio of the real part of the second highest eigenvalue to the real part of its first, provided by the equation
\begin{equation}
\Lambda^{(\mathbf{A})}:=\frac{\Re(\lambda_2)}{\Re{(\lambda_1)}}=\frac{\Re(\lambda_2)}{{b}}.
\label{EQ:UI}
\end{equation}
We consider the core index to be \emph{small} if $\Lambda^{(\mathbf{A})} <\nicefrac{1}{2}$, \emph{critical} if $\Lambda^{(\mathbf{A})}=\nicefrac{1}{2}$, and \emph{large} if $\Lambda^{(\mathbf{A})}> \nicefrac{1}{2}$.

For our covariance matrix, we define 
\begin{align*}
\hat{\mathbf{P}} :=\mathbf{I}_k - \mathbf{v}_1\,\mathbf{1}_k^T
\end{align*}
and the symmetric matrix 
\begin{equation}
\mathbf{B}^{(\mathbf{A})} := s^{-2}\mathbf{A}^T  \bm{\mathcal{Q}}\mathbf{A},
\label{EQ: BSym3}
\end{equation}
where $\bm{\mathcal{Q}}= s(s-1) \mathbf{v}_1^T\mathbf{v}_1+s\, \mathbf{diag}(\mathbf{v}_1)$.
Lastly, we use $e^{{s} \mathbf{A}} := \sum_{k=0}^{\infty} \frac{1}{k!} \big({m}\mathbf{A}\big)^k$ for all $m\in \mathbb{R}$, to define the integral
\begin{equation}
\label{EQ:SigAInt3}
\bm{\Sigma}^{(\mathbf{A})} := b \,\int_0^\infty e^{{m} \mathbf{A}^T}\hat{\mathbf{P}}^T\mathbf{B}^{(\mathbf{A})}\hat{\mathbf{P}}\,e^{{m} \mathbf{A}} e^{-\mathscr{b} m}\, dm.
\end{equation}

For urn models with small core index, \cite{Sparks23} shows that the mean and covariance matrix both grow asymptotically linear in nature, and by a modification of techniques presented in \cite{Janson, Janson2020}, the following multivariate central limit theorem can be derived.  

\begin{theorem}
\label{Thm:AffMCLT}
Consider a $k$-color irreducible affine urn that is balanced with ${b}>0$ and involving samples of size $m \ge 1$. Let $\Lambda^{(\mathbf{A})} < \nicefrac{1}{2}$. Then, for samples without replacement, for large $n$,

$$\E [\mathbf{X}_n]= n{b} \mathbf{v}_1 + \mathbf{o}(n^{\nicefrac{1}{2}}),\textnormal{ and}$$

$$\frac{1}{n}\bm{\Sigma}_n  \to  \bm{\Sigma}^{(\mathbf{A})},$$ 
where $\bm{\Sigma}^{(\mathbf{A})}$ is a limiting matrix defined in (\ref{EQ:SigAInt3}), implying $\bm{\Sigma}_n  = n{b} \bm{\Sigma}^{(\mathbf{A})}+\mathbf{o}(n)$. Furthermore, as $n\to\infty$,
\begin{align*}
\frac{1}{\sqrt{n}}(\mathbf{X}_{n}-{b} n \mathbf{v}_1) \convD \, \normal_k \Big(\mathbf{0}_k, \,\bm{\Sigma}^{(\mathbf{A})} \Big),
\end{align*}
\end{theorem}

\section{Application to the Global Containment Profile}
\label{Sec:global}
We now wish to examine the global profile of the hyperrecursive tree, determined by
a raw count of the number of vertices at a particular containment level. Note that this measure cannot be determined without looking at all the vertices in the entire hyperrecursive tree, which leads to a shift from \emph{local} to \emph{global} in \cite{Sparks2}. We define $X_{n,i}^{(\theta)}$ to be the number of vertices contained in exactly $i$ hyperedges at step $n$. In Figure~\ref{Fig:hyper}, we see that $X_{2,1}^{(3)}=2, X_{2,2}^{(3)}=2$, $X_{2,3}^{(3)}=1,\textnormal{ and }X_{2,i}^{(3)}=0$ 
for all $i \ge 4$.

In \cite{Sparks2}, deriving the mean and covariance of first $k$ levels of containment involves lengthy recursive formulas.  However, we may approach this problem instead by constructing a similar random vector $\mathbf{Y}_{n}^{(\theta)}$. Given $k\in \mathbb{N}$, we set the row vector $\mathbf{Y}_n^{(\theta)}$ such that $Y_{n, i}^{(\theta)}=X_{n, i}^{(\theta)}$ for $i\in[k]$ and $Y_{n, (k+1)}^{(\theta)}=\sum_{i=k+1}^{\infty}X_{n, i}^{(\theta)}$, which measures the number of vertices with containment level greater than $k$. With this modification, the global profile of the hyperreccursive tree then relates to an affine urn problem of $(k+1)$ colors with sample size~$(\theta-1)$. 

\begin{proposition}
\label{Prop: YAffine}
The random vector $\mathbf{Y}_n^{(\theta)}$ as defined above is a class of 1-balanced, $(k+1)$-color affine urn models with a sample of size $s=(\theta-1)$.
\end{proposition}
\begin{proof}
Consider a sample from the random vector $\mathbf{Y}_n^{(\theta)}$ such that $Y_{n, i}^{(\theta)}=X_{n, i}^{(\theta)}$ for $i\in[k]$ and $Y_{n, (k+1)}^{(\theta)}=\sum_{i=k+1}^{\infty}X_{n, i}^{(\theta)}$, for size $\theta$ and $n\in \mathbb{N}$. Let $s_i$ be the sample of vertices of containment level $i$. At each step, we add a vertex of containment level~1, and for each vertex sampled from containment level $i\le k$, we remove one vertex from that containment level and add it to containment level $(i+1)$. For a vertices measured by  $Y_{n, (k+1)}$, there is no change. Thus, for a sample  $\mathbf{s}= (s_1, s_2, ... ,s_k, s_{k+1})$, we obtain a replacement vector of 
$$\mathbf{a}_\mathbf{s}=(1-s_1, s_1-s_2, ..., s_{k-1}-s_{k},s_{k}).$$  
From this result, we construct a replacement matrix for representing samples which all come from the same color:
\begin{align}
\mathbf{A}=\small\begin{matrix}(\theta-1) \mathbf{e}_1 \cr
                               (\theta-1) \mathbf{e}_2 \cr
			 \vdots\cr
                               	(\theta-1) \mathbf{e}_{k}\cr
					(\theta-1) \mathbf{e}_{k+1}
\end{matrix}
\begin{pmatrix} -(\theta-2) & (\theta-1)&0&\cdots&0&0 \cr
                               	1& -(\theta-1) & (\theta-1)&\cdots&0&0 \cr
			\vdots&\vdots&\vdots&\ddots&\vdots&\vdots \cr
                               	1 & 0&0&\cdots&-(\theta-1)&(\theta-1) \cr
                               	1 & 0&0&\cdots&0&0 \cr
\end{pmatrix}.
\label{EQ:HTA}
\end{align}
To prove both linearity and that $\mathbf{A}$ is the core matrix for an affine urn, we show that $\mathbf{a}_{\mathbf{s}} = \sum_{i=1}^{k} \frac{s_i}{s}\mathbf{a}_{s\,\mathbf{e}_i}$. Given a vector sample $\mathbf{s}$, we obtain
\begin{align*}
\sum_{i=1}^{k+1} \frac{s_i}{s}\mathbf{a}_{s\,\mathbf{e}_i}=&\, \frac{s_1}{\theta-1}
(-(\theta-2), (\theta-1), 0, \dots, 0)\\
&+\frac{s_2}{\theta-1}(1, -(\theta-1), (\theta-1),\dots,0)\\
&+\cdots+\frac{s_{k}}{\theta-1}(1,0, \dots, -(\theta-1), (\theta-1))\\
&+\frac{s_{k+1}}{\theta-1}(1, 0, 0,\dots,0)\\[1 em]
{}
=&\, \Bigl(-\frac{\theta-2}{\theta-1}s_1 + \sum_{i=2}^{k} \frac{s_i}{\theta-1} + \frac{s_k}{\theta-1} ,s_1-s_2, \dots, s_{k-1}-s_{k},s_{k} \Bigr).
\end{align*}
Since $s_k= (\theta-1)- \sum_{i=1}^{k} s_i$, the first entry reduces to 
\begin{align*}
-\frac{\theta-2}{\theta-1}s_1 + \sum_{i=2}^{k} \frac{s_i}{\theta-1} + \frac{(\theta-1)- \sum_{i=1}^{k} s_i}{\theta-1}
= 1 - \frac{\theta-2}{\theta-1}s_1 - \frac{1}{\theta-1}s_1 = 1-s_1,
\end{align*}
therefore proving the urn's affinity.
\end{proof}

\begin{corollary}
\label{Cor:EigenHT}
The core matrix defined as (\ref{EQ:HTA}) has eigenvalues 1 and $(1-\theta)$ with multiplicities 1 and $k$, respectively.
\end{corollary}

\subsection{Stochastic Recurrences and the First Two Moments}
\label{Subsec:stoch}
To observe how the different containment levels interact,
we investigate the mean of the row vector $\mathbf{X}_{n}^{(\theta)}$ along with its covariance matrix and asymptotic distribution. From here, we establish a central limit theorem for the vertices at the first $k$ levels of containment. Let~$Q_{n,i}^{(\theta)}$ be the number of vertices at containment level $i\in [k]$ that appear in the sample chosen to construct the $n^{th}$ hyperedge. Using (\ref{EQ:HTA}), we see that
\begin{equation}
X_{n,1}^{(\theta)} = X_{n-1,1}^{(\theta)} - Q_{n,1}^{(\theta)} + 1,
\label{Eq:stoch1}
\end{equation}
\begin{equation}
X_{n,j}^{(\theta)} = X_{n-1,j}^{(\theta)} - Q_{n,j} ^{(\theta)}+  Q_{n,j-1}^{(\theta)}
\label{Eq:stoch2}
\end{equation}
for each $j=2, \dots, k.$ 

Thus, each vertex at containment level 1 in the sample upgrades to containment level 2, each vertex at containment level 2 in the sample upgrades to containment level 3, and so on, with the newly added vertex $n$ arriving at containment level 1. The pair of stochastic equations~(\ref{Eq:stoch1})--(\ref{Eq:stoch2})
is sufficient to determine the means and the quadratic order moments asymptotically.

\begin{theorem}
\label{Prop:mean}
Let $X_{n,i}^{(\theta)}$ be the number of vertices  contained in exactly $i$ hyperedges, for $i\ge1$, of a hyperrecursive tree of edge size $\theta$. Then, the asymptotic solution for the mean is given by
\begin{align*}
\E\big[X_{n,i}^{(\theta)}\big]= \frac{(\theta-1)^{i-1}}{\theta^i}(n+\theta) + \mathcal{O}\Big(\frac{\ln^{i-1}(n)}{n}\Big).
\end{align*}
Furthermore, the exact values for the first two containment levels are 
\begin{align}
\label{EQ:Means}
\small\begin{pmatrix}
     \E\big[X_{n,1}^{(\theta)}\big] \medskip\\
     \E\big[X_{n,2}^{(\theta)}\big]
   \end{pmatrix}  =&\,\small \begin{pmatrix} \displaystyle
     \frac 1 \theta\,  n + 1 + (\theta-1) \, \Gamma (\theta) \frac {\Gamma(n+1)}{
                \Gamma(n+\theta)} \medskip \\
     \displaystyle  \frac {(\theta-1)} {\theta^2} ( n + \theta) 
              +   \Gamma (\theta) \frac {\Gamma(n+1)}{
                \Gamma(n+\theta)} \Big((\theta-1)^2 H_n - \frac {\theta -1}
                       \theta \Big) 
   \end{pmatrix}.
\end{align}
\end{theorem} 

\begin{proof}
While the asymptotic means may be found through the principal left eigenvector of $\mathbf{A}$ from (\ref{EQ:HTA}), finding definite bounds for the asymptotic orders are best achieved by a recursive approach. Using the process performed in \cite{Sparks2}, the exact values for the first two containment levels are found to be (\ref{EQ:Means}), which produce the asymptotic results as stated above.

Now, let the asymptotic result be true for all $n \in \mathbb{N}$ for $i=1, \dots, r$. By way of induction, we use this assumption to bootstrap results for $i=r+1$. Then, we have
\begin{align*}
\E\big[X_{n,r+1}^{(\theta)}\,|\,\mathbb{F}_{n-1} \big]=&\, \E\big[X_{n-1,r+1}^{(\theta)}\,|\,\mathbb{F}_{n-1} \big]- \E\big[Q_{n,r+1}^{(\theta)}\,|\,\mathbb{F}_{n-1} \big]+ \E\big[Q_{n,r}^{(\theta)}\,|\,\mathbb{F}_{n-1} \big] \\[1 em]
{}
=&\, \frac{n}{n+\theta-1} \E\big[X_{n-1,r+1}^{(\theta)}\,|\,\mathbb{F}_{n-1} \big]+ \frac{\theta-1}{n+\theta-1} \E\big[X_{n-1,r}^{(\theta)}\,|\,\mathbb{F}_{n-1} \big]\\[1 em]
{}
=&\, \frac{n}{n+\theta-1} \E\big[X_{n-1,r+1}^{(\theta)}\,|\,\mathbb{F}_{n-1} \big]+ \frac{(\theta-1)^{r}}{\theta^{r}}+ \mathcal{O}\Big(\frac{\ln^{r-1}(n)}{n^2}\Big).
\end{align*}
For containment level to be solved, it has to await for the solution of the previous level so that it may be bootstrapped into the equation. These equations follow the standard form 
\begin{equation}
y_n = g_n y_{n-1} + h_n,
\label{Eq:standard}
\end{equation}
with solution
$$y_n = \sum_{i=1}^n h_i\prod_{j=i+1}^n g_j + y_0 \prod_{j=1}^n g_j. $$

From (\ref{Eq:standard}), we obtain the our solution.
\begin{align*}
\E\big[X_{n,r+1}^{(\theta)}\,|\,\mathbb{F}_{n-1} \big]=&\, 0 + \sum_{\ell=1}^{n} \Big(\prod_{j=\ell+1}^{n} \Big(\frac{j}{j+\theta-1}\Big)\Big) \Big(\frac{(\theta-1)^{r}}{\theta^{r}}+ \mathcal{O}\Big(\frac{\ln^{r-1}(\ell)}{\ell^2}\Big) \Big) \\[1 em]
{}
=&\, \frac{\Gamma(n+1)}{\Gamma(n+\theta)}  \Bigg(\frac{(\theta-1)^{r}}{\theta^{r}} \sum_{\ell=1}^{n}\frac{\Gamma(\ell+\theta)}{\Gamma(\ell+1)} +\sum_{\ell=1}^{n}\frac{\Gamma(\ell+\theta)}{\Gamma(\ell+1)}\mathcal{O}\Big(\frac{\ln^{r-1}(\ell)}{\ell^2}\Big)   \Bigg)\\[1 em]
{}
=&\,\frac{\Gamma(n+1)}{\Gamma(n+\theta)}  \Big(\frac{(\theta-1)^{r}}{\theta^{r}}\Big) \Big(\frac{\Gamma(n+\theta+1)}{\theta\,\Gamma(n+1)} -\frac{\Gamma(\theta+1)}{\theta\,\Gamma(1)}\Big) \\
& + \frac{\Gamma(n+1)}{\Gamma(n+\theta)}\sum_{\ell=1}^{n}\frac{\Gamma(\ell+\theta)}{\Gamma(\ell+1)}\mathcal{O}\Big(\frac{\ln^{r-1}(\ell)}{\ell^2}\Big)  .
\end{align*}
Note that 
\begin{align*}
\frac{\Gamma(n+1)}{\Gamma(n+\theta)}\sum_{\ell=1}^{n}\frac{\Gamma(\ell+\theta)}{\Gamma(\ell+1)}\mathcal{O}\Big(\frac{\ln^{r-1}(\ell)}{\ell^2}\Big)=&\,
\begin{cases}\mathcal{O}\big(\frac{\ln^{r}(\ell)}{\ell}\big) & \textnormal{if } \theta=2; \medskip\\
\mathcal{O}\big(\frac{\ln^{r-1}(\ell)}{\ell}\big), & \textnormal{if } \theta>2, 
\end{cases}
\end{align*}
and so we get
\begin{align*}
\E\big[X_{n,r+1}^{(\theta)}\,|\,\mathbb{F}_{n-1} \big] =&\, \frac{(\theta-1)^{r}}{\theta^{r+1}}(n+\theta) +\mathcal{O}(n^{1-\theta}) + \mathcal{O}\Big(\frac{\ln^{r}(\ell)}{\ell}\Big) \\
=&\, \frac{(\theta-1)^{r}}{\theta^{r+1}}(n+\theta)+ \mathcal{O}\Big(\frac{\ln^{r}(\ell)}{\ell}\Big),
\end{align*}
thus proving our result for all $i\in\mathbb{N}$.
\end{proof}
\begin{corollary}
\label{Cor:EVHRT}
Consider the random vector $\mathbf{Y}_n^{(\theta)}$ defined in Proposition \ref{Prop: YAffine}, for a given $k\in\mathbb{N}$. Then, the principal left eigenvector of its affine core matrix $\mathbf{A}$ is $\mathbf{v}_1$, such that
\begin{align*}
{v}_{1,i}=\begin{cases} \frac{(\theta-1)^{{i-1}}}{\theta^{i}}, & \textnormal{if } i \in [k]; \medskip\\
\big(\frac{\theta-1}{\theta}\big)^k, & \textnormal{if } i=k+1.
\end{cases}
\end{align*}
\end{corollary}

With our mean vector determined, we now develop our associated covariance matrix. In \cite{Sparks2}, the first two containment levels were determined using recurrence relations, but with our link to affine urns established, we generalize here a solution for the first $k$ containment levels, where $k \in \mathbb{N}$.   

\begin{theorem}
\label{Theo:variance}
Let $X_{n,i}^{(\theta)}$ be the number of vertices contained in exactly $i$ hyperedges in a hyperrecursive tree with parameter $\theta$ at age $n$, for $i\in [k]$. Let ${\bf \Sigma}_n^{(\theta)} $ be the corresponding covariance matrix for the global profile of hyperrecursive tree, and let $\mathbf{Y}_n^{(\theta)}$ be defined in Proposition 5.1, with $\mathbf{A}$ as the core matrix associated with the affine urn as defined in (\ref{EQ:HTA}). Then,
\begin{align*}
\bm{\Sigma}_n^{(\theta)}  = n \bm{\Sigma}_\infty^{(\theta)}+\mathbf{o}(n),
\end{align*} 
where $\bm{\Sigma}_\infty^{(\theta)}$ is the upper-left $k\times k$ block of 
$\bm{\Sigma}^{(\mathbf{A})}$, 
the limiting matrix defined in (\ref{EQ:SigAInt3}).
\end{theorem}

\begin{proof}
In \cite{Sparks2}, the covariance matrix for the first two containment levels is obtained via application of the stochastic
recurrences~(\ref{Eq:stoch1}) and~(\ref{Eq:stoch2}). However, by linking the structure to an affine urn, we attain a solution as follows. Given $k\in \mathbb{N}$, consider the random vector $\mathbf{Y}_n^{(\theta)}$ defined in Proposition \ref{Prop: YAffine}. From this construction, we form the core matrix $\mathbf{A}$ as provided in (\ref{EQ:HTA}). We attain its eigenvalues from Corollary \ref{Cor:EigenHT} and form the principal left eigenvector as provided in Corollary \ref{Cor:EVHRT}. Then, we obtain the Jordan decomposition $\mathbf{A}= \mathbf{T}\mathbf{{J}}\mathbf{T}^{-1}$, where
$$\mathbf{{J}}_{1}=(1),    ~~~~~~~~~~~~~~~~\mathbf{{J}}_{(1-\theta)}=\small\begin{pmatrix} -(\theta-1) & 1&0&\cdots&0&0 \cr
                               	0& -(\theta-1) & 1&\cdots&0&0 \cr
			\vdots&\vdots&\vdots&\ddots&\vdots&\vdots \cr
                               	0 & 0&0&\cdots&-(\theta-1)&1 \cr
                               	0 & 0&0&\cdots&0&-(\theta-1) \cr
\end{pmatrix}. $$

Since the core index $\Lambda^{(\mathbf{A})}<\nicefrac{1}{2}$, we apply Theorem \ref{Thm:AffMCLT} to conclude that the covariance matrix $\bm{\Sigma}_n^{(\mathbf{A})} \sim n  \,\bm{\Sigma}^{(\mathbf{A})}$ for large $n$. Thus, the covariance matrix for the first $k$ containment levels is asymptotic to the matrix $n\,\bm{\Sigma}_\infty^{(\theta)}$, where $\bm{\Sigma}_\infty^{(\theta)}$ is the upper-left $k\times k$ block from $\bm{\Sigma}^{(\mathbf{A})}$. 
\end{proof} 

\begin{example}
To illustrate the use of Theorem \ref{Theo:variance}, we will focus on the first $k=3$ containment levels an examine the variability of  $(X_{n,1}^{(\theta)},X_{n,2}^{(\theta)}, X_{n,3}^{(\theta)})$. Specifically, $\bm{\Sigma}_\infty^{(\theta)}$ is a $3\times3$ symmetric matrix with

\begin{gather}
\bm{\Sigma}_\infty^{(\theta)}[1,1]= \frac{(\theta -1)^2}{\theta ^2 (2 \theta -1)} ~~~~~~~~~~~
\bm{\Sigma}_\infty^{(\theta)}[1,2]= -\frac{(\theta -1)^2 \left(\theta ^2+2 \theta -1\right)}{\theta ^3 (2 \theta -1)^2}\nonumber\\[1 ex]
{}
\bm{\Sigma}_\infty^{(\theta)}[1,3]=-\frac{(\theta -1)^2 \left(\theta ^4+\theta ^3-7 \theta ^2+5 \theta -1\right)}{\theta ^4 (2 \theta -1)^3} \nonumber\\[1 ex]
{}
\bm{\Sigma}_\infty^{(\theta)}[2,2]= \frac{(\theta -1)^2 \left(6 \theta ^4-6 \theta ^3+8 \theta ^2-5 \theta +1\right)}{\theta ^4 (2 \theta -1)^3}\nonumber\\[1 ex]
{}
\bm{\Sigma}_\infty^{(\theta)}[2,3]=-\frac{(\theta -1)^2 \left(3 \theta ^6-16 \theta ^4+32 \theta ^3-24 \theta ^2+8 \theta -1\right)}{\theta ^5 (2 \theta -1)^4}\nonumber \\[1 ex]
{}
\bm{\Sigma}_\infty^{(\theta)}[3,3]=\small{\frac{(\theta -1)^2 \left(26 \theta ^8-74 \theta ^7+112 \theta ^6-152 \theta ^5+170 \theta ^4-121 \theta ^3+50 \theta ^2-11 \theta +1\right)}{\theta ^6 (2 \theta -1)^5}}
\label{cov3}
\end{gather}

To begin, we produce a four-component random vector $\mathbf{Y}^{(\theta)}_n$ such that $\mathbf{Y}^{(\theta)}_{n,i}=\mathbf{X}^{(\theta)}_{n,i}$ for $i\in[3]$ and $\mathbf{Y}^{(\theta)}_{n,4}=\sum_{i=4}^{\infty}\mathbf{X}^{(\theta)}_{n,i}.$  From there, our corresponding affine urn model with $m=(\theta-1)$ drawings has an initial state of $\mathbf{Y}^{(\theta)}_0=(\theta,\,0,\, 0,\, 0)$ and core matrix 

$$\mathbf{A}=\small\begin{pmatrix}-(\theta-2) & (\theta-1)& 0&0 \\
                               	1 & -(\theta-1)&(\theta-1)&0 \\
			1& 0&-(\theta-1)&(\theta-1)\\
					1& 0&0&0
\end{pmatrix}.$$

Here, we have a $1$-balanced affine urn, where the eigenvalues of $\mathbf{A}$ are $\lambda_1=1$ and \\$\lambda_2=-(\theta-1)=\lambda_3=\lambda_4$. The leading left eigenvector is $\mathbf{v}_1=(\frac{1}{\theta}, \,\frac{(\theta-1)}{\theta^2},\, \frac{(\theta-1)^2}{\theta^3}, \,(\frac{\theta-1}{\theta})^3)$. 

Note that $\mathbf{A}$ is not diagonalizable, so our analysis requires using the Jordan decomposition $\mathbf{A}= \mathbf{T}{\mathbf{J}}\mathbf{T}^{-1}$ as described in (\ref{EQ:JD}), where
$$\mathbf{T}={\small 
\begin{pmatrix}
 1 & -(\theta-1) & 1 & 0 \\
 1 & 1 & -\frac{\theta}{\theta-1} & \frac{1}{\theta-1} \\
 1 & 1 & 0 & -\frac{\theta}{(\theta-1)^2} \\
 1 & 1 & 0 & 0 \\
\end{pmatrix}}, ~~~~~~~{\mathbf{J}}={\small \begin{pmatrix}
 1 & 0 & 0 & 0 \\
 0 & -(\theta-1) & 1 & 0 \\
 0 & 0 &-(\theta-1)  &1 \\
 0 & 0 & 0 &-(\theta-1) 
\end{pmatrix}}.$$ 

Again, we set $\mathbf{P}_{\lambda_1}=\mathbf{1}_k^T \mathbf{v}_1$, $\mathbf{B}^{(\mathbf{A})}$ as in (\ref{EQ: BSym3}), and let 
$$e^{s \mathbf{A}}=\mathbf{T}\,{\small \begin{pmatrix}
 e^s & 0 & 0 & 0 \\
 0 & e^{-(\theta-1) s} & e^{-(\theta-1) s} s & \frac{1}{2} e^{-(\theta-1) s} s^2 \\
 0 & 0 & e^{-(\theta-1) s} & e^{-(\theta-1) s} s \\
 0 & 0 & 0 & e^{-(\theta-1) s} 
\end{pmatrix}}
\,\mathbf{T}^{-1}.$$

As $\Lambda^{(\mathbf{A})}<\nicefrac{1}{2}$, we apply Theorem \ref{Thm:AffMCLT} and get the limiting matrix 
\begin{align*}
\bm{\Sigma}^{(\mathbf{A})}=\int_0^\infty e^{{s} \mathbf{A}^T}\hat{\mathbf{P}}^T\mathbf{B}^{(\mathbf{A})}\hat{\mathbf{P}}\,e^{{s} \mathbf{A}} e^{-s}\, ds,
\end{align*}
a matrix with the upper-left $k\times k$ square given by (\ref{cov3}).
\end{example}

\subsection{Multivariate Convergence in Distribution}
In \cite{MahmoudSmythe012}, multivariate convergence in distribution was attained directly through the Cram\'{e}r-Wold Device. However, through the lens of affine urn models, we utilize~$\mathbf{Y}_n^{(\theta)}$ from the proof in Proposition \ref{Prop: YAffine} and apply Theorem \ref{Thm:AffMCLT} to establish convergence in distribution through the methods obtained via Section \ref{Affine}. 

\begin{theorem}Let $\mathbf{X}_{n,i}^{(\theta)}$ be the vector measuring the number of vertices contained in exactly $i$ hyperedges, for $i\in [k]$, of a hyperrecursive tree with hyperedges of size $\theta$ at age $n$. Let $\bm{\mu}_\infty^{(\theta)}$ be the $k$-dimensional vector such that $\mu_{\infty, i}^{(\theta)}=\frac{(\theta-1)^{i-1}}{\theta^i}$ for $i\in [k]$, and $\mathbf{\Sigma}_\infty^{(\theta)}$ be defined as in Theorem \ref{Theo:variance}. Then, as $n\to\infty$, we have
$$\frac{1}{\sqrt{n}}(\mathbf{X}_{n}^{(\theta)} -n\bm{\mu}_\infty^{(\theta)})\,  \convD \, \normal_k \big(\mathbf{0}_k,\mathbf{\Sigma}_\infty^{(\theta)} \big).$$
\end{theorem}

\begin{proof}
First, consider the random vector $\mathbf{Y}_n^{(\theta)}$ as defined in Proposition \ref{Prop: YAffine}. Then, from Corollary \ref{Cor:EigenHT}, the core matrix $\mathbf{A}$ corresponding to this $(k+1)$-vector has a core index is less than $\nicefrac{1}{2}$. By Theorem \ref{Thm:AffMCLT}, we apply Corollary \ref{Cor:EVHRT} to conclude 
\begin{align*}
\frac{1}{\sqrt n}(\mathbf{Y}_{n}^{(\theta)} -n\mathbf{v}_1)\, \convD \, \normal_{k+1} \big(\mathbf{0}_{k+1},\mathbf{\Sigma}^{(\mathbf{A})} \big).
\end{align*}
Since $\mathbf{Y}_n^{(\theta)}$ converges to a $(k+1)$-variable normal distribution, we conclude that its marginal distribution $\mathbf{X}_n^{(\theta)}$ converges to a $k$-variable normal distribution, specifically,
\begin{align*}
\frac{1}{\sqrt n}(\mathbf{X}_{n}^{(\theta)} -n\bm{\mu}_\infty^{(\theta)})\,  \convD \, \normal_k \big(\mathbf{0}_k,\mathbf{\Sigma}_\infty^{(\theta)} \big).
\end{align*}
\end{proof}

The following corollary gives us the the precise values for convergence in distribution for the first three containment levels.

\begin{corollary}Let $(X_{n,1}^{(\theta)},X_{n,2}^{(\theta)}, X_{n,3}^{(\theta)})$ be the number of vertices contained in the first three hyperedges of a hyperrecursive tree
with hyperedges of size $\theta$ at age $n$. Then, as $n\to\infty$, 
we have
$$\frac{1}{\sqrt n}\Big(\big(X_{n,1}^{(\theta)},X_{n,2}^{(\theta)}, X_{n,3}^{(\theta)}\big)-\big(\frac{n}{\theta}, \frac{(\theta-1)n}{\theta^2}, \frac{(\theta-1)^2 n}{\theta^3}\big)\Big)  \convD \, \normal_3 \big((0,0,0),\mathbf{\Sigma}_\infty^{(\theta)} \big),$$ where $\mathbf{\Sigma}_\infty^{(\theta)}$ is equal to (\ref{cov3}).
\end{corollary}

\begin{remark}
In the special case of $\theta=2$, the hyperrecursive tree is the standard uniform recursive tree. In this case, $X_{n,1}^{(2)}$ is the count of the leaves in the tree and $X_{n,k}^{(2)}$ is the number of nodes of outdegree $(k-1)$. The above corollary recovers the result in~\cite{MahmoudSmythe012} and generalizes it. Furthermore, this process extends beyond the first three containment levels by generalizing results outlined in~\cite{MahmoudSmythe012}.
\end{remark}

\section{Simulations on First Three Containment Levels}
We finish this section by examining the rate of convergence on $\big(X_{n,1}^{(\theta)}, X_{n,2}^{(\theta)},X_{n,3}^{(\theta)}\big)$ and provide simulations on the first two levels of containment. We analyze the cases
$\theta=2, 3, 4, 5$. To allow sufficient mixing, we perform 2000 urn draws for each simulation, record the number of balls at containment levels 1 and 2, and replicate the process 1000 times. Using these results, we provide estimates for the means and the covariance matrix and assess the trivariate normality of $\big(X_{n,1}^{(\theta)}, X_{n,2}^{(\theta)},X_{n,3}^{(\theta)}\big)$ by using the Henze-Zirkler's (HZ) Test for Multivariate Normality.

Table \ref{table:HRT} provides estimates that are quite close to our exact means and asymptotic variances and covariances, with HZ test statistics which do not reject the assumption of trivariate normality. In Figure \ref{fig:figTX12} below, we obtain simulated distributions for $X_{2000,1}^{(\theta)}, X_{2000,2}^{(\theta)}$ and $X_{2000,3}^{(\theta)}$ that appear to be approximately normal in nature as demonstrated by the similarity between the distribution curve of the kernel (black) to that of the theoretical Gaussian curve (red), supporting the table and theory provided. It is worthy to note that the simulation result suggest that there is a relatively quick rate of convergence with respect to the limiting distribution for the vector of counts for each of these containment levels.

\begin{table}[h]
\caption{Simulations of $\big(X_{2000,1}^{(\theta)}, X_{2000,2}^{(\theta)},  X_{2000,3}^{(\theta)}\big)$ with $1000$ Replications.}\label{table:HRT}%
\begin{tabular}{@{}l|l|l@{}}
\toprule
$\theta$ & Theoretical Values & Simulation Results \\
\midrule
$\theta$=2 & $\bm{\mu}_{\infty}=\begin{pmatrix} 0.5,& 0.25,& 0.125\end{pmatrix}$ &$\hat{\bm{\mu}}_{\infty}\approx\begin{pmatrix}0.500,& 0.250,& 0.125\end{pmatrix}$\\
&&\\
&$\bm{\Sigma}_{\infty}\approx\begin{pmatrix}
 0.083 & -0.097 & -0.012 \\
 -0.097 & 0.164 & -0.043 \\
 -0.012 & -0.043 & 0.091 \\
\end{pmatrix}$&$\hat{\bm{\Sigma}}^{(\mathcal{R})}_{\infty}\approx\begin{pmatrix}
0.079 &-0.094& -0.012\\
-0.094&  0.160 &-0.040 \\
 -0.012 &-0.040&  0.089\\
\end{pmatrix}$ \\ 
&&\\
&&HZ$=0.671; p=0.958$\\
\hline
$\theta$=3 & $\bm{\mu}_{\infty}\approx\begin{pmatrix}0.333,& 0.222,& 0.148\end{pmatrix}$ &$\hat{\bm{\mu}}_{\infty}\approx\begin{pmatrix}0.334,& 0.223,& 0.148\end{pmatrix}$\\
&&\\
&$\bm{\Sigma}_{\infty}\approx\begin{pmatrix}
 0.089 & -0.083 & -0.023 \\
 -0.083 & 0.151 & -0.041 \\
 -0.023 & -0.041 & 0.113 \\
\end{pmatrix}$&$\hat{\bm{\Sigma}}^{(\mathcal{R})}_{\infty}\approx\begin{pmatrix}
  0.092& -0.082& -0.026 \\
-0.082  &0.150 &-0.046 \\
-0.026 &-0.046 & 0.125\\
\end{pmatrix}$ \\ 
&&\\
&&HZ$=1.018; p=0.118$\\
\hline
$\theta$=4 & $\bm{\mu}_{\infty}\approx\begin{pmatrix}0.25, &0.188, &0.141\end{pmatrix}$ &$\hat{\bm{\mu}}_{\infty}\approx\begin{pmatrix}0.251,& 0.187,& 0.141\end{pmatrix}$\\
&&\\
&$\bm{\Sigma}_{\infty}\approx\begin{pmatrix}
 0.08 & -0.066 & -0.023 \\
 -0.066 & 0.129 & -0.036 \\
 -0.023 & -0.036 & 0.109 \\
\end{pmatrix}$&$\hat{\bm{\Sigma}}^{(\mathcal{R})}_{\infty}\approx\begin{pmatrix}
 0.087& -0.071& -0.027  \\
 -0.071 & 0.138 &-0.038  \\
-0.027& -0.038 & 0.115 
\end{pmatrix}$ \\ 
&&\\
&&HZ$=1.028; p=0.104$\\
\hline
$\theta$=5 & $\bm{\mu}_{\infty}\approx\begin{pmatrix} 0.2, &0.16,& 0.128\end{pmatrix}$ &$\hat{\bm{\mu}}_{\infty}\approx\begin{pmatrix}0.201,& 0.161,& 0.128\end{pmatrix}$\\
&&\\
&$\bm{\Sigma}_{\infty}\approx\begin{pmatrix}
 0.071 & -0.054 & -0.021 \\
 -0.054 & 0.112 & -0.031 \\
 -0.021 & -0.031 & 0.1 \\
\end{pmatrix}$&$\hat{\bm{\Sigma}}^{(\mathcal{R})}_{\infty}\approx\begin{pmatrix}
0.074& -0.052& -0.022 \\
 -0.052  &0.107 &-0.032 \\
-0.022& -0.032  &0.096 \\
\end{pmatrix}$ \\ 
&&\\
&&HZ$=0.706; p=0.915$\\
\botrule
\end{tabular}
\end{table}

\noindent

\begin{figure}[H]
\centering
{\fbox{$\theta=2$}}
 \includegraphics[width=130mm, scale=0.2]{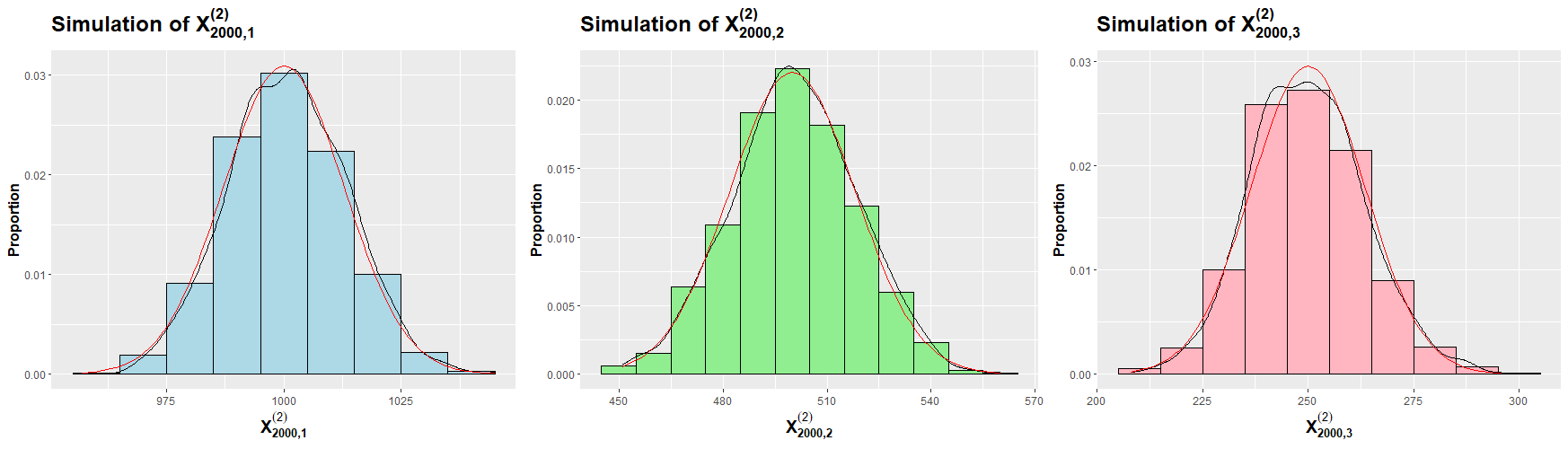}
{\fbox{$\theta=3$}}
 \includegraphics[width=130mm, scale=0.2]{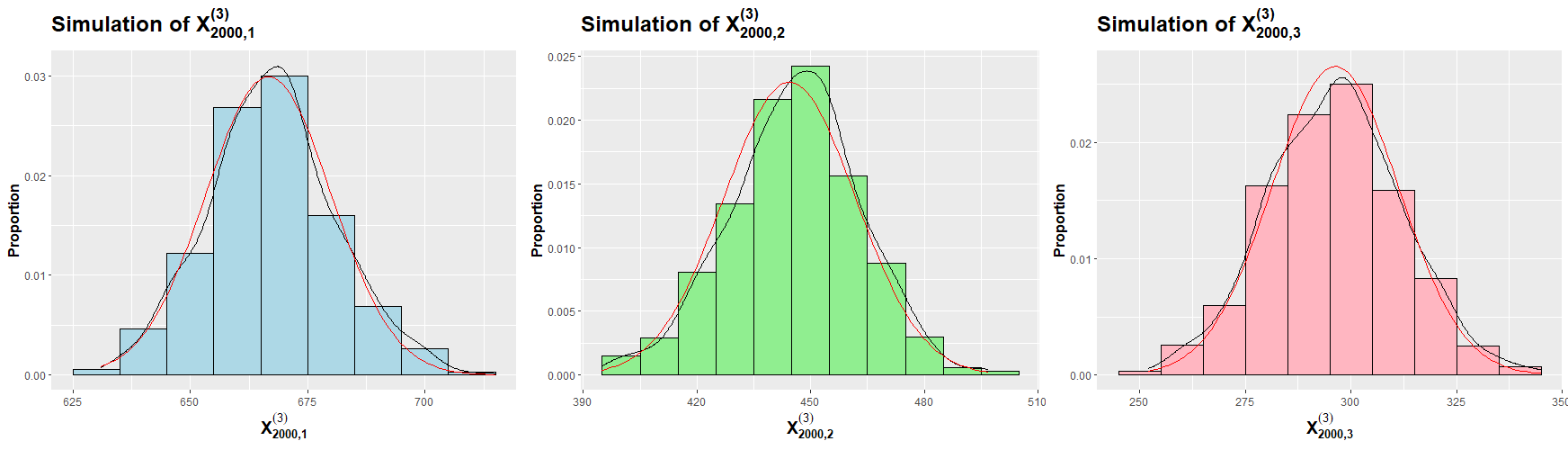}
{\fbox{$\theta=4$}}
 \includegraphics[width=130mm, scale=0.2]{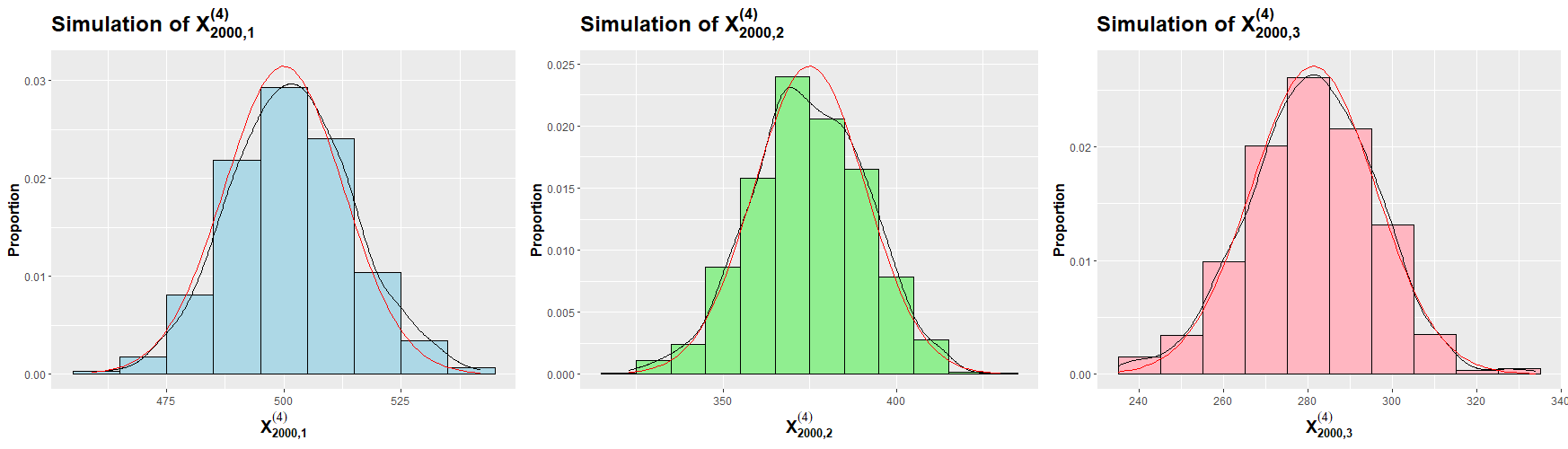}
{\fbox{$\theta=5$}}
 \includegraphics[width=130mm, scale=0.2]{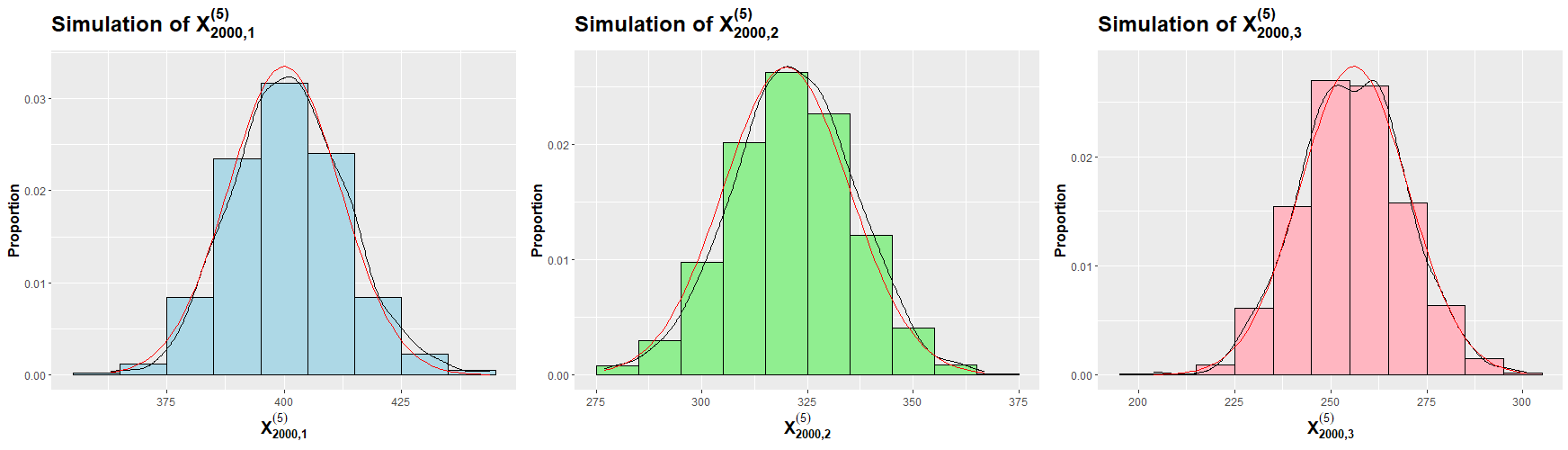}
\caption{Simulations of $\big(X_{2000,1}^{(\theta)}, X_{2000,2}^{(\theta)},X_{2000,3}^{(\theta)} \big)$ with $1000$ Replications}
\label{fig:figTX12}
\end{figure}

\section{Declarations}
\subsection{Funding Declaration}
There is no funding declarations to report.
\subsection{Availability of Data and Materials Declaration}
The simulated data can be acquired using the R code found in Appendix A, with the multivariate hypergeometric distribution attained via the \emph{extraDistr} package.
\subsection{Competing Interest Declaration}
There are no competing interests to report.

\bibliography{sn-bibliography}

\begin{appendices}
\section{R Code for Hyperrecursive Trees} 
Algorithm 1 below is the R programming code that will produce the simulation results provided in Section 5.
\begin{algorithm}
\caption{R Code for Generating Simulation for the First Three Global Containment Levels of a Hyperrecursive Tree}
\label{HRTR}
\small
\begin{verbatim}

MHyperurnD<-function(th,n){
  A<-t(matrix(c(-(th-2),(th-1),0,0,
                1,-(th-1),(th-1),0,
                1,0,-(th-1),(th-1),
                1,0,0,0),nrow=4))
  m=th-1
  x <- c(th,0,0,0)
  k<-length(x)
  r<-m+k-1
  t<-k-1
  E <- matrix(c(replicate((k+1)*choose(r,t),0)), ncol=k+1)
  E[1:choose(r,t),2]=r
  for(i in 3:(k+1)){E[1,i]=k+1-i}
  for(h in 2:choose(r,t)){ s<-0
    while(s<k){
      q<-k+1-s
      if((E[h-1,q-1]-E[h-1, q])==1){
        E[h,q]=s
        s=s+1 }
      else{
        for(j in 2:q){E[h,j]=E[h-1,j]}
        E[h,k+1-s]=E[h-1,k+1-s]+1
        s=k}  }}
  C <- matrix(c(replicate(k*choose(r,t),0)), ncol=k)
  for(p in 1:choose(r,t)){
    C[p, 1]=r-E[p, 2]-1
    C[p, k] = E[p,k+1]
    for(z in 2:k-1){
      C[p,z] = E[p,z+1]-E[p,z+2]-1 } }
  M <- matrix(c(replicate(k*choose(r,t),0)), ncol=k)
  for(i in 1:choose(r,t)){
    M[i,]= C[i,]%*%A/m  }
  tt=0
  while(tt < n){  
    H <- rmvhyper(1, x, m)
    for(i in 1:choose(r,t)){
      if(isTRUE(all.equal(C[i,],H[1:k]))){
        x=x+M[i,]}
      else{x=x}}
    tt=tt+1} 
  return(x)}
set.seed(8801)
theta2=replicate(1000,MHyperurnD(2,2000))
set.seed(9501)
theta3=replicate(1000,MHyperurnD(3,2000))
set.seed(4102)
theta4=replicate(1000,MHyperurnD(4,2000))
set.seed(4706)
theta5=replicate(1000,MHyperurnD(5,2000))
\end{verbatim}
\end{algorithm}
\end{appendices}
\end{document}